\documentclass[12pt,twoside,reqno]{amsart}
\usepackage[colorlinks=true,citecolor=blue]{hyperref}
\usepackage{mathptmx, amsmath, amssymb, amsfonts, amsthm, mathptmx, enumerate, color,mathrsfs}
\theoremstyle{plain}
\newtheorem{thm}{\protect\theoremname}
\theoremstyle{definition}
\newtheorem{problem}[thm]{\protect\problemname}
\theoremstyle{remark}
\newtheorem{rem}[thm]{\protect\remarkname}

\usepackage[a4paper]{geometry}
\usepackage{graphicx}
\usepackage{float}
\usepackage{multirow}
\usepackage{epstopdf}
\usepackage{multicol}
\usepackage{epstopdf}
\usepackage{esint}
\usepackage{lineno}

\usepackage[ruled,linesnumbered]{algorithm2e}

\SetKwBlock{Repeat}{repeat}{}

\usepackage{setspace}
\setdisplayskipstretch{3}

\newtheorem{theorem}{Theorem}[section]
\newtheorem{lemma}[theorem]{Lemma}

\newtheorem{condition}{Condition}
\newtheorem{assumption}{Assumption}
\theoremstyle{definition}
\newtheorem{definition}[theorem]{Definition}

\providecommand{\problemname}{Problem}
\providecommand{\remarkname}{Remark}
\providecommand{\theoremname}{Theorem}

\makeatother

\providecommand{\problemname}{Problem}
\providecommand{\remarkname}{Remark}
\providecommand{\theoremname}{Theorem}

\numberwithin{equation}{section}

\DeclareMathOperator{\Id}{Id}

\DeclareMathOperator{\Fix}{Fix}

\DeclareMathOperator*{\argmin}{argmin}

\DeclareMathOperator{\modul}{mod}

\begin{document}
\setcounter{page}{1}

\vspace*{2.0cm}
\title[Solving nonlinear equations]
{Approaches to iterative algorithms for solving nonlinear equations
with an application in tomographic absorption spectroscopy}
\author[F. J. Aragón-Artacho et al.]{Francisco J. Aragón-Artacho$^{1}$, Weiwei Cai$^{2}$, Yair Censor$^{3}$, Aviv Gibali$^{4,*}$, Chongyuan Shui$^{2}$ and David Torregrosa-Belén$^{1}$}
\maketitle
\vspace*{-0.6cm}

\begin{center}
{\footnotesize

$^{1}$Department of Mathematics, University of Alicante, Alicante,
Spain\\
$^{2}$Department of Mechanical Engineering, Shanghai Jiao Tong University,
Shanghai, China\\
$^{3}$Department of Mathematics, University of Haifa, Haifa, Israel\\
$^{4}$Applied Mathematics Department, HIT Holon Institute of Technology,
Holon, Israel\\ \bigskip

\textbf{This work is dedicated to the memory of Professor Hedy Attouch}
}\end{center}

\vskip 4mm {\footnotesize \noindent {\bf Abstract.}
In this paper we propose an approach for solving systems of nonlinear
equations without computing function derivatives. Motivated by the
application area of tomographic absorption spectroscopy, which is
a highly-nonlinear problem with variables coupling, we consider a
situation where straightforward translation to a fixed point problem
is not possible because the operators that represent the relevant
systems of nonlinear equations are not self-mappings, i.e., they operate
between spaces of different dimensions.

To overcome this difficulty we suggest an ``alternating common fixed
points algorithm'' that acts alternatingly on the different vector
variables. This approach translates the original problem to a common
fixed point problem for which iterative algorithms are abound and
exhibits a viable alternative to translation to an optimization problem,
which usually requires derivatives information. However, to apply
any of these iterative algorithms requires to ascertain the conditions
that appear in their convergence theorems.

To circumvent the need to verify conditions for convergence, we propose
and motivate a derivative-free algorithm that better suits the tomographic
absorption spectroscopy problem at hand and is even further improved
by applying to it the superiorization approach. This is presented
along with experimental results that demonstrate our approach.


  }

\renewcommand{\thefootnote}{}
\footnotetext{ $^*$Corresponding author.
\par
E-mail address: francisco.aragon@ua.es (F. J. Aragón-Artacho), (W. Cai) cweiwei@sjtu.edu.cn, yair@math.haifa.ac.il (Y. Censor), avivgi@hit.ac.il (A. Gibali), scy1949@sjtu.edu.cn (C. Shui) and david.torregrosa@ua.es (D. Torregrosa-Belén).
\par
Received: January 24, 2024. Revised: May 10, 2024.

\rightline {\tiny   \copyright  2022 Communications in Optimization Theory}}

\renewcommand{\thefootnote}{\arabic{footnote}}

\section{Introduction\label{sec:Introduction}}

Solving systems of nonlinear equations has long been and still is
a fundamental problem in mathematics with countless real-world applications
that demand efficient methods to accomplish the task, see, e.g., the
books by Ortega and Rheinboldt \cite{OR-book-2000} and by Rheinboldt
\cite{rheinboldt1998methods}. A Google Scholar search on this topic
returns, not surprisingly, numerous entries. A central approach is
based on transforming the systems of nonlinear equations to an optimization
problem and using methods from that field, mainly methods for unconstrained
minimization which are commonly geared toward convex functions, see,
e.g., Boyd and Vandenberghe's book \cite{boyd2004convex}.

In practical situations it is, however, often the case that no derivatives
of the functions that comprise the system of equations are known and
even if derivatives exist, they are not calculable. This hinders the
applicability of minimization methods and a variety of heuristics
have been developed such as, for example, the simulated annealing
method, consult, e.g., the information on the auto-generated ScienceDirect
Webpage on this subject\footnote{https://www.sciencedirect.com/topics/engineering/simulated-annealing-algorithm.}.

An alternative route is to transform the system of nonlinear equations
into an operator equation such that every solution to the system is
a fixed point of the operator and vice versa. This approach hinges
on the premise that the associated operator is a self-mapping from
a space into itself. In this paper we investigate the situation when
this does not hold and the operator maps one space into a different
space. For this scenario we propose an alternating fixed point approach.
Specifically, we define a suitable family of fixed point operators
that allows constructing an \textbf{alternating common fixed points
algorithm} that can handle the problem.

However, in order to guarantee the convergence of the algorithm, some
quite restrictive conditions on these operators are required. In particular,
the approach seems not to be practical for the real-world application
in tomographic absorption spectroscopy in which we are interested
here.

Facing this kind of problems and motivated by the alternating fixed
point approach described above, we take a deeper look at the properties
of the equations at hand and suggest to use them in a different way.
In particular, we address the case in which the equations depend on
two variables, and the dependence on one of them is linear. For this
problem, we propose a derivative-free algorithm which also acts alternatingly
on each of the variables but makes use of descent directions of the
summands of the least squares problem associated to the system of
equations. We call it the \textbf{descent pairs algorithm} (DPA).
The iterative nature of our descent pairs algorithm enables us to
introduce a priori conditions into its iterative process. We do this
via the \textbf{superiorization methodology}. This methodology works
by taking an iterative algorithm, investigating its \textbf{perturbation
resilience}, and then, using proactively such permitted perturbations,
it forces the perturbed algorithm to do something useful in addition
to what it is originally designed to do. We present a numerical validation
of the descent pairs algorithm, with and without superiorization,
for a real-world application in tomographic absorption spectroscopy.
In this field, problems as the one presented here have been widely
studied and a great variety of algorithms have been employed to tackle
them. Our experiments show that the approach proposed here yields
results that compete well with those obtained by these methods under
similar conditions. As a general comment we care to mention that common
fixed point iterative methods and related problems are a field of
vigorous research with many new directions and developments, see,
e.g., \cite{mewomo1,mewomo2,mewomo3} and Cegielski's book \cite{Ceg-book}.

The paper is structured as follows. In Section \ref{sect:1}, we briefly
recall the well-known fixed point approach for tackling systems of
nonlinear equations and present our alternating common fixed points
algorithm adapted for the case in which the operator is not a self-mapping.
In Section~\ref{sec:descentdirections}, the particular instance
of the problem with the linear relation is tackled. We present our
descent pairs algorithm and mathematically support the idea behind
the algorithmic scheme. Section \ref{Sect:TAS}, contains a broad
view of tomographic absorption spectroscopy theory and in the last
subsection, we include experiments that successfully demonstrate the
good performance of our descent pairs algorithm, with and without
superiorization, in this field.

\section{Problem formulation and an alternating common fixed point algorithm}

\label{sect:1} We are interested in solving a system of nonlinear
equations as formulated next.
\begin{problem}
\label{prob:nonl-syst}Let $\mathbb{R}^{M}$ be the Euclidean $M$-dimensional
space. Consider a family of functions $\beta_{k}:\mathbb{R}^{M}\times\mathbb{R}^{M}\to\mathbb{R}^{M}$
and vectors $b^{k}=(b_{j}^{k})_{j=1}^{M}\in\mathbb{R}^{M}$, for $k\in\{1,2,\ldots,W\}$.
\begin{equation}
\text{Find }x,y\in\mathbb{R}^{M}\text{ such that }\beta_{k}(x,y)=b^{k},\quad k=1,2,\ldots,W.\label{e:absyst}
\end{equation}
\end{problem}

The motivation to look at this problem comes from a real-world application
in \textbf{tomographic absorption spectroscopy} (TAS), see, e.g.,
Dai et al.~\cite{2018JQSRT}, that we discuss later in Section~\ref{Sect:TAS}.
In some real-world applications, including TAS, the following condition
prevails.

\begin{condition} \label{cond:no-derivatives} No derivatives of
the functions $\beta_{k}$ are known and even if they exist they are
not calculable. \end{condition}

A common approach for solving systems of nonlinear equations in such
a situation consists of translating the system into a \textbf{fixed
point problem} (FPP), see, e.g., Combettes~\cite{framework2020}
and Combettes and Pesquet~\cite{combettes-eusipco-2018}.

To illustrate this approach, consider a system of nonlinear equations
\begin{equation}
\gamma_{j}(z)=c_{j},\text{ }\ j=1,2,\ldots,M,\label{eq:generic-prob}
\end{equation}
where $\gamma_{j}:\mathbb{R}^{M}\rightarrow\mathbb{R}$, for $j=1,2,\ldots,M,$
are given real-valued functions and $c=(c_{j})_{j=1}^{M}\in\mathbb{R}^{M}$
is a given vector. Denoting by $\varGamma:\mathbb{R}^{M}\rightarrow\mathbb{R}^{M}$
the operator
\begin{equation}
\varGamma:=\left(\begin{array}{c}
\gamma_{1}\\
\gamma_{2}\\
\vdots\\
\gamma_{M}
\end{array}\right),\label{eq:operator-1}
\end{equation}
the system (\ref{eq:generic-prob}) becomes an \textbf{operator equation}
\begin{equation}
\varGamma(z)=c.
\end{equation}
Since $\varGamma$ is a self-mapping, it is well-known how to translate
the system (\ref{eq:generic-prob}) into a problem of finding a point
in the set of fixed points $\Fix(T):=\{x\in\mathbb{R}^{M}\mid T(x)=x\}$
of a suitably defined operator $T$ (see, e.g., Berinde's book \cite[Chapter 8, page 179]{berinde-book-2007}).
This is done by considering the operator $T:\mathbb{R}^{M}\rightarrow\mathbb{R}^{M}$
given by
\begin{equation}
T:=c+(\Id-\varGamma),\label{e:trick}
\end{equation}
where $\Id:\mathbb{R}^{M}\rightarrow\mathbb{R}^{M}$ is the identity
operator, that is, $\Id(x)=x$. Then, it is easy to see that for a
point $x^{*}\in\mathbb{R}^{M}$
\begin{equation}
T(x^{*})=x^{*}\text{ \ if and only if }\varGamma(x^{*})=c
\end{equation}
and one can solve the system (\ref{eq:generic-prob}) by solving the
fixed point problem for the operator~$T$.

The difficulty in applying this fixed point approach to the problem
posed in \eqref{e:absyst} lies in the fact that $\beta_{k}:\mathbb{R}^{M}\times\mathbb{R}^{M}\to\mathbb{R}^{M}$
are not self-mappings. To the best of our knowledge, the challenge
of adapting the fixed point theory to this setting has not been attended
before. Therefore, we create here an alternating common fixed point
algorithm that applies the approach of fixed point theory alternatingly
to each of the two vector variables $x$ and $y$ of the functions
$\beta_{k}$ of (\ref{e:absyst}).

The adaptation of the fixed point theory to our problem works as follows.
For any pair $x,y\in\mathbb{R}^{M}$ and for any $k=1,2,\ldots,W$,
define the operators
\begin{equation}
B_{k}\left(\begin{array}{c}
x\\
y
\end{array}\right):=\left(\begin{array}{c}
\beta_{k}(x,y)\\
\beta_{k}(x,y)
\end{array}\right).
\end{equation}
Then $B_{k}:\mathbb{R}^{M}\times\mathbb{R}^{M}\rightarrow\mathbb{R}^{M}\times\mathbb{R}^{M}$
are self-mappings and, for all $k=1,2,\ldots,W$,
\begin{equation}
\beta_{k}(x,y)=b^{k}\Longleftrightarrow B_{k}\left(\begin{array}{c}
x\\
y
\end{array}\right)=\left(\begin{array}{c}
b^{k}\\
b^{k}
\end{array}\right).
\end{equation}
Now use the technique of Equation \eqref{e:trick} and define, for
all $k=1,2,\ldots,W,$ the operators
\begin{equation}
T_{k}:=\left(\begin{array}{c}
b^{k}\\
b^{k}
\end{array}\right)+\left(\Id-B_{k}\right),
\end{equation}
i.e.,
\begin{equation}
T_{k}\left(\begin{array}{c}
x\\
y
\end{array}\right):=\left(\begin{array}{c}
b^{k}\\
b^{k}
\end{array}\right)+\left(\left(\begin{array}{c}
x\\
y
\end{array}\right)-B_{k}\left(\begin{array}{c}
x\\
y
\end{array}\right)\right).\label{e:Tkdouble}
\end{equation}
Then, for all $k=1,2,\ldots,W,$
\begin{equation}
T_{k}\left(\begin{array}{c}
x^{\ast}\\
y^{\ast}
\end{array}\right)=\left(\begin{array}{c}
x^{\ast}\\
y^{\ast}
\end{array}\right)\Longleftrightarrow\beta_{k}(x^{\ast},y^{\ast})=b^{k},
\end{equation}
and finding a solution to Problem \ref{prob:nonl-syst} is equivalent
to solving the common fixed point problem (CFPP)
\begin{equation}
\text{Find }\left(\begin{array}{c}
x^{*}\\
y^{*}
\end{array}\right)\in{\displaystyle \bigcap_{k=1}^{W}\Fix(T_{k}).}\label{e:CFFP}
\end{equation}
There is an extensive literature of algorithms devoted to solving
the CFPP in \eqref{e:CFFP}, see, e.g., Zaslavski's book \cite{Zaslavski-book-2016}.
As an example, the algorithm presented here in Algorithm \ref{alg:AFP-1}
is the cyclic version of the \textbf{almost cyclic sequential algorithm}
(ACSA) for the common fixed point problem, in Censor and Segal \cite[Algorithm 5]{SegalCensor},
which is, in turn, a special case of the algorithm in Combettes \cite[Algorithm 6.1]{combettes-quasi-2001}.

For the operators defined in \eqref{e:Tkdouble}, Algorithm \ref{alg:AFP-1}
leads to the proposed \textbf{alternating common fixed points algorithm}.
It is the iterative process, that starts with an arbitrary $\left(\begin{array}{c}
x^{0}\\
y^{0}
\end{array}\right)\in\mathbb{R}^{M}\times\mathbb{R}^{M},$ and then, for all $\ell\geq0,$ updates according to
\begin{equation}
\left(\begin{array}{c}
x^{\ell+1}\\
y^{\ell+1}
\end{array}\right)=\left(\begin{array}{c}
x^{\ell}\\
y^{\ell}
\end{array}\right)+\lambda_{\ell}\left(\left(\begin{array}{c}
b^{i(\ell)}\\
b^{i(\ell)}
\end{array}\right)-\left(\begin{array}{c}
\beta_{i(\ell)}(x^{\ell},y^{\ell})\\
\beta_{i(\ell)}(x^{\ell},y^{\ell})
\end{array}\right)\right).\label{eq:scheme0}
\end{equation}
However, the convergence theorem for ACSA, see, e.g., \cite[Theorem 6]{SegalCensor},
requires additional assumptions on the involved operators. They should
be ``directed operators''\footnote{Directed operators are nowadays called ``cutters'', see \cite[Subsection 2.1.3]{Ceg-book}.}
such that, for all $k,$ $T_{k}-\Id$ should be closed at $0$. An
(uninteresting) example verifying these properties are the projection
operators onto the diagonal subspace.

{\SetAlgoNoLine
\begin{algorithm}[h]
\textbf{Initialization:} Let $\{T_{k}:\mathbb{R}^{M}\rightarrow\mathbb{R}^{M}\}_{k=1}^{W}$
be a finite family of operators$.$ Set $\{\lambda_{\ell}\}_{\ell=0}^{\infty}$
a sequence in $[0,2]$. Choose $x^{0}\in\mathbb{R}^{M}$\; set {$\ell=0$}\;
\Repeat{ set {$i(\ell)=\ell\,\modul W+1$}\; set {$x^{\ell+1}=x^{\ell}+\lambda_{\ell}\left(T_{i(\ell)}(x^{\ell})-x^{\ell}\right)$}\;
} \caption{The Cyclic Sequential Algorithm (CSA) for the common fixed point problem}
\label{alg:AFP-1}
\end{algorithm}}

The difficulty of having operators whose block coordinates are identically
defined, as is the case for the operators in \eqref{e:Tkdouble},
and verifying the assumptions of the convergence theorem for ACSA,
lead us to search for an alternative approach to handle the problem.
In the next section we develop a different algorithmic scheme which
is inspired by the iterative process \eqref{eq:scheme0}.

\section{The descent pairs algorithm}

\label{sec:descentdirections}

In this section we propose and motivate a derivative-free algorithm
for tackling Problem \ref{prob:nonl-syst} in the case when the operators
fulfill the following two assumptions. We denote the set of vectors
whose components are all different from zero by
\begin{equation}
\mathbb{R_{\mathrm{\neq0}}^{\mathrm{\mathit{M}}}}:=\{x\in\mathbb{R}^{M}\mid x_{j}\neq0\;\textup{for}\;\textup{all}\;j=1,2,\ldots,M\}.
\end{equation}
\begin{assumption}\label{asumpt:1} For every $k\in{\{1,2,\ldots,W\}}$,
the operators $\beta_{k}$ depend linearly on the second variable
and nonlinearly on the first variable, such that they can be expressed,
in matrix form, as
\begin{equation}
\beta_{k}(x,y)={\rm diag}\left(\widetilde{\beta}_{k}(x)\right)y,\quad k=1,2,\ldots,W,\label{e:assumpt1}
\end{equation}
where $\widetilde{\beta}_{k}:\mathbb{R}^{M}\to\mathbb{R}_{\mathrm{\neq0}}^{\mathrm{\mathit{M}}}$
and ${\rm diag}(u)$ denotes the diagonal matrix with the vector $u$
along its diagonal. \end{assumption}

For our second assumption, we need to introduce the following definitions.

\begin{definition} (i) \textbf{Descent direction} (see, e.g., \cite[Definition 5.1]{Aragon2019}).
Given a function $g:\mathbb{R}^{M}\to\mathbb{R}$ which is differentiable
at some vector $x\in\mathbb{R}^{M}$, then a vector $v\in\mathbb{R}^{M}$
is called a \textbf{descent direction for $g$ at the point $x$}
if
\begin{equation}
g'(x;v)=\lim_{t\to0}\frac{g(x+tv)-g(x)}{t}=\nabla g(x)^{T}v<0.
\end{equation}

(ii) \textbf{Descent pair}. Given a mapping $f:\mathbb{R}^{M}\to\mathbb{R}^{N}$
and a vector $u\in\mathbb{R}^{N}$ we say that $(f,u)$ is a \textbf{descent
pair}\emph{ }if, for every $x\in\mathbb{R}^{M}$, the vector
\begin{equation}
v=v(x):=u-f(x)
\end{equation}
is a descent direction at the point $x$ for the function $g:\mathbb{R}^{M}\to\mathbb{R}$
defined as
\begin{equation}
g(\cdot):=\frac{1}{2}\|u-f(\cdot)\|^{2}.
\end{equation}
\end{definition}

\begin{assumption}\label{assumpt:2} There exists an index $t\in{\{1,2,\ldots,W\}}$
such that $b^{t}\in\mathbb{R}_{\mathrm{\neq0}}^{\mathrm{\mathit{M}}}$
and such that for all $k=1,2,\ldots,W$, the pairs $(f_{k},u^{k})$,
with

\begin{equation}
f_{k}(\cdot):={\rm diag}\left(\widetilde{\beta}_{t}(\cdot)\right)\widetilde{\beta}_{k}(\cdot)\text{ and }u^{k}:={\rm diag}\left(b^{t}\right)^{-1}b^{k},
\end{equation}
and where the ``$-1$'' power represents matrix inversion, are descent
pairs. Moreover, for each $x\in\mathbb{R}^{M}$, the pairs $(h_{k},b^{k})$,
where
\begin{equation}
h_{k}(\cdot):=\beta_{k}(x,\cdot),
\end{equation}
are descent pairs for all $k=1,2,\ldots,W$. \end{assumption}

In the tomographic absorption spectroscopy problem, modeled as Problem
\ref{prob:nonl-syst} and studied in subsection \ref{sect:TAS3} below,
Assumptions \ref{asumpt:1} and \ref{assumpt:2} are fulfilled.

The next lemma shows that, although Problem \ref{prob:nonl-syst}
is a system which depends on both variables $x$ and $y$, the Assumption
\ref{asumpt:1} on the operators $\beta_{k}$ implies that the variable
vector $y$ does not interfere with the suitability of the variable
vector $x$ for solving the system.

\begin{lemma}\label{lem:1} Consider a family of operators $\beta_{k}:\mathbb{R}^{M}\times\mathbb{R_{\mathrm{\neq0}}^{\mathrm{\mathit{M}}}}\to\mathbb{R_{\mathrm{\neq0}}^{\mathrm{\mathit{M}}}}$,
for $k\in\{1,2,\ldots,W\}$, for which Assumption \ref{asumpt:1}
holds. Then, a point $x^{*}\in\mathbb{R}^{M}$ belongs to a solution
pair $(x^{*},y^{*})$ of the system of equations given by \eqref{e:absyst}
if and only if $x^{*}$ is a solution of the system
\begin{equation}
{\rm diag}\left(\widetilde{\beta}_{t}(x)\right)^{-1}\widetilde{\beta}_{k}(x)={\rm diag}(b^{t})^{-1}b^{k},\;\textup{for}\;\textup{all}\;k=1,2,\ldots,W,\;k\neq t,\label{e:systx}
\end{equation}
with $t\in{\{1,2,\ldots,W\}}$ such that $b^{t}\in\mathbb{R_{\mathrm{\neq0}}^{\mathrm{\mathit{M}}}}$.
\end{lemma}
\begin{proof}
Let $(x^{*},y^{*})$ be a solution of the system given by \eqref{e:absyst}.
Then, by Assumption~\ref{asumpt:1}, for all $k,t\in{\{1,2,\ldots,W\}},$
and using elementary rules for matrix inversion and diagonal matrices,
we have
\begin{equation}
\begin{aligned}\beta_{k}(x^{*},y^{*})=b^{k} & \Leftrightarrow{\rm diag}\left(\widetilde{\beta}_{k}(x^{*})\right)y^{*}=b^{k}\\
 & \Leftrightarrow{\rm diag}\left({\rm diag}\left(\widetilde{\beta}_{t}(x^{*})\right)y^{*}\right)^{-1}{\rm diag}\left(\widetilde{\beta}_{k}(x^{*})\right)y^{*}={\rm diag}\left(b^{t}\right)^{-1}b^{k}\\
 & \Leftrightarrow{\rm diag}\left(\widetilde{\beta}_{t}(x^{*})\right)^{-1}{\rm diag}\left(y^{*}\right)^{-1}{\rm diag}\left(\widetilde{\beta}_{k}(x^{*})\right)y^{*}={\rm diag}\left(b^{t}\right)^{-1}b^{k}\\
 & \Leftrightarrow{\rm diag}\left(\widetilde{\beta}_{t}(x^{*})\right)^{-1}\widetilde{\beta}_{k}(x^{*})={\rm diag}\left(b^{t}\right)^{-1}b^{k}.
\end{aligned}
\end{equation}
Thus, $x^{*}$ is a solution for the single vector variable system
\eqref{e:systx}.
\end{proof}
In Algorithm \ref{a:abstract} below, which we call the \textbf{descent
pairs algorithm},\textbf{ }we present a method for tackling a system
of the form \eqref{e:absyst} that obeys Assumptions \ref{asumpt:1}
and \ref{assumpt:2}. The motivation of the algorithm is as follows.
In order to obtain a solution pair $(x^{*},y^{*})$ to the system
\eqref{e:absyst} we generate two separate iterative sequences. The
first sequence is denoted by $\{x^{\ell}\}_{\ell=0}^{\infty}$ and
employs Lemma \ref{lem:1} to find $x^{*}$ as a solution to the system
\eqref{e:systx}. To achieve this, we consider the least squares problem
associated with \eqref{e:systx}, which is
\begin{equation}
\argmin_{x\in\mathbb{R}^{M}}{\displaystyle \sum_{k=1,k\neq t}^{W}g_{k}(x)={\displaystyle \sum_{k=1,k\neq t}^{W}\frac{1}{2}\left\Vert {\rm diag}\left(\widetilde{\beta}_{t}(x)\right)^{-1}\widetilde{\beta}_{k}(x)-{\rm diag}\left(b^{t}\right)^{-1}b^{k}\right\Vert ^{2}.}}\label{e:lspx}
\end{equation}
Lines 4 to 8 of the algorithm generate $x^{\ell+1}$ from $x^{\ell}$
by sequentially updating $x^{\ell}$ with constant step-size line
searches for each of the summands in \eqref{e:lspx}, in the descent
direction given by the descent pair provided by Assumption \ref{assumpt:2}.

The purpose of the second sequence $\left\{ y^{\ell}\right\} _{\ell=0}^{\infty}$
is to find the point $y^{*}$. Lines 9 to 13 show how to obtain $y^{\ell+1}$
from $y^{\ell}$ and $x^{\ell+1}$. Indeed, $y^{\ell}$ is sequentially
updated by performing a constant step-size line search for each of
the functions
\begin{equation}
h_{k}(y):=\frac{1}{2}\|b^{k}-\beta_{k}(x^{\ell+1},y)\|^{2},\quad k=1,2,\ldots,W,
\end{equation}
in the descent direction determined by the descent pair given by Assumption
\ref{assumpt:2}.\smallskip{}

{\SetAlgoNoLine
\begin{algorithm}[h]
\textbf{Initialization.}~Choose $x^{0}$, $y^{0}\in\mathbb{R}^{M}$.
Set an index $t\in{\{1,2,\ldots,W\}}$ in compliance with Assumption
\ref{assumpt:2} and pick real fixed relaxation parameters $\lambda_{x}>0$
and $\lambda_{y}>0$\; Set {$\ell=0$}\; \Repeat{ Set {$x^{\ell,0}=x^{\ell}$}\;
\For{$q=1,2,\ldots,W$}{ $x^{\ell,q}=x^{\ell,q-1}+\lambda_{x}\left({\rm diag}(b^{t})^{-1}b^{q}-{\rm diag}\left(\widetilde{\beta}_{t}(x^{\ell,q-1})\right)^{-1}\widetilde{\beta}_{q}(x^{\ell,q-1})\right)$\;
} Set {$x^{\ell+1}=x^{\ell,W}$}\; Set {$y^{\ell,0}=y^{\ell}$}\;
\For{$q=1,2,\ldots,W$}{ $y^{\ell,q}=y^{\ell,q-1}+\lambda_{y}\left(b^{q}-\beta_{q}(x^{\ell+1},y^{\ell,q-1})\right)$\;
} Set $y^{\ell+1}=y^{\ell,W}$\; } \caption{The Descent Pairs Algorithm (DPA).}
\label{a:abstract}
\end{algorithm}
}

In Section \ref{Sect:TAS}, we illustrate the good performance of
the proposed Algorithm~\ref{a:abstract} in a demonstrative numerical
experiment arising from a real-world problem in tomographic absorption
spectroscopy, in which our algorithm shows a competitive potential
vis-{à}-vis with state-of-the-art methods in the field.

\section{An experimental demonstration in tomographic absorption spectroscopy}

\label{Sect:TAS}

\subsection{Introduction to tomographic absorption spectroscopy}

Absorption spectroscopy is a popular technique for gas sensing which
can simultaneously retrieve thermophysical properties such as temperature,
species concentration and pressure \cite{CAI20171}. When a laser
beam penetrates a \textbf{region of interest} (ROI) filled with gaseous
medium, its intensity is attenuated due to the absorption of the gas
molecules along the \textbf{line-of-sight} (LOS).

The aim of absorption spectroscopy is to obtain information of the
gaseous medium by measuring spectrum-specified absorbance. The Beer-Lambert
law, which relates the attenuation of light to the properties of the
material through which the light is traveling, provides the relationship
between the gas properties (i.e., pressure, temperature and concentration)
and the absorbance, see \cite{Lin:10},

\begin{equation}
b(\nu):=\ln(I_{0}(\nu)/I_{t}(\nu))=\int PS(\nu,x)\,y\,dL.\label{eq:beer-lamb}
\end{equation}

The absorbance $b(\nu)$ is the logarithmic ratio of the incident
$I_{0}(\nu)$ and transmitted $I_{t}(\nu)$ intensities for the absorption
line at wavelength $\nu$. In (\ref{eq:beer-lamb}), $P$ is the pressure,
which is supposed to be a known constant, and $S$ is the line strength
function, whose value depends on the wavelength $\nu$ and the temperature
$x$. For each wavelength of interest, the line function $S$ has
an approximately negative exponent relation with the reciprocal of
the temperature, i.e. $1/x$. The concentration of the absorbing species
is $y$ and $L$ is the length of the LOS.

In practical applications, the gas properties are usually non-uniform
along the LOS, therefore, tomography is needed to enable the spatial
resolution of absorption spectroscopy. In order to represent a mathematically
tractable model, full discretization is done, meaning that the light
spectrum is discretized into a finite number of wavelengths, the ROI
is discretized into a finite number of pixels/voxels and the external
light sources are discretized into a finite number of individual beams.
This modeling of tomographic absorption spectroscopy (TAS) is illustrated
in Figure \ref{fig:1}.

\begin{figure}[h]
\centering \includegraphics[scale=0.65]{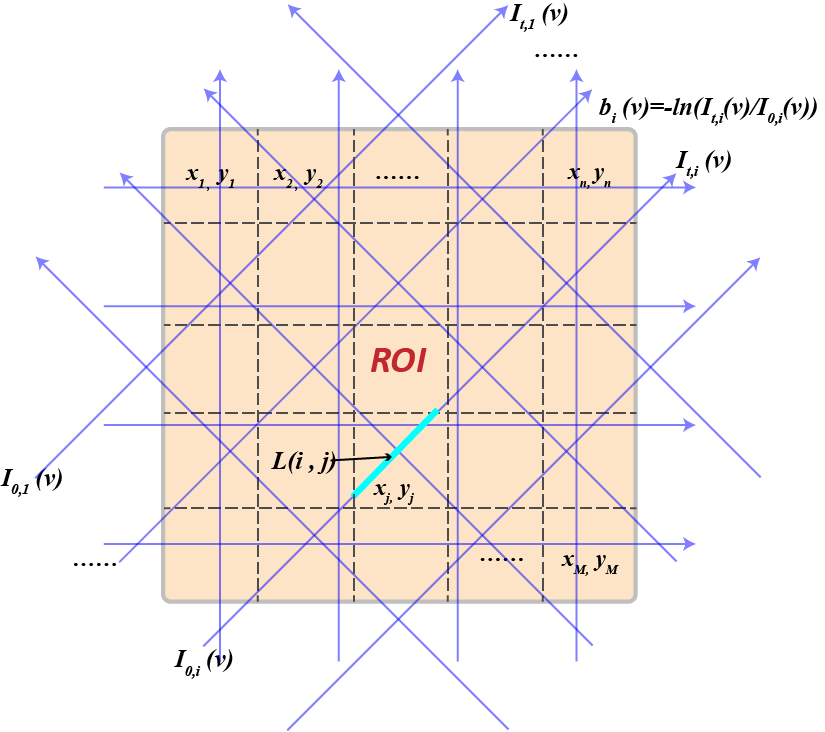}
\caption{Illustration of the fully-discretized model for tomographic absorption
spectroscopy.}
\label{fig:1}
\end{figure}

The discretized model for TAS is obtained as follows. First, the absorbance
is measured along a finite number of beams, indexed by $i=1,2,\ldots,N,$
covering the ROI. The light spectrum is discretized into a finite
number of wavelengths, indexed by $k=1,2,\ldots,W.$ When $N$ beams
are employed and $W$ absorption features are probed in the measurements,
the relationship between the temperature $x,$ the concentration of
the absorbing species $y$ and the measured absorbance $b$ can be
expressed as

\begin{equation}
b_{i}^{k}=\mathop{\intop_{L_{i}}}\alpha_{\nu_{k}}(x,y)\,dL,\;k=1,2,\ldots,W,\,i=1,2,\ldots,N,\label{e:int}
\end{equation}
where $b_{i}^{k}$ is the absorbance along the $i$-th beam with length
$L_{i}$ at the $k$-th wavelength $\nu_{k}$, and $\alpha_{\nu_{k}}$
is the absorption coefficient, related to the wavelength $\nu_{k}$
and the local values of $x$ and $y$.

Finally, we assume that the ROI is discretized and consists of a two-dimensional
square, meshed into $M=n\times n$ square pixels, indexed by $j=1,2,\ldots,M$.
Thus, denoting by $L_{i,j}$ the length of intersection of the $i$-th
light beam within the $j$-th pixel, equation \eqref{e:int} is reinterpreted
in its fully-discretized form as

\begin{equation}
b_{i}^{k}=\mathop{\sum_{j=1}^{M}}\alpha_{\nu_{k}}(x_{j},y_{j})\cdot L_{i,j}=\mathop{\sum_{j=1}^{M}}\alpha_{k}(x_{j},y_{j})\cdot L_{i,j},\label{e:bik}
\end{equation}
for $k=1,2,\ldots,W,$ $i=1,2,\ldots,N,$ $j=1,2,\ldots,M$. The $\alpha_{k}:\mathbb{R}\times\mathbb{R}\to\mathbb{R}$
are nonlinear operators which measure the absorption coefficient at
the $k$-th wavelength. Note that after the discretization, $x$ and
$y$ are redefined as vectors in $\mathbb{R}^{M}$.

Repeating equation \eqref{e:bik} for all the beams at the $k$-th
wavelength, a set of equations is formulated in matrix form as

\begin{equation}
b^{k}=\left(\begin{array}{c}
b_{1}^{k}\\
b_{2}^{k}\\
\vdots\\
b_{i}^{k}\\
\vdots\\
b_{N}^{k}
\end{array}\right)=\left(\begin{array}{ccccc}
L_{1,1} & \cdots & L_{1,j} & \cdots & L_{1,M}\\
\vdots & \ddots & \vdots & \ddots & \vdots\\
L_{i,1} & \cdots & L_{i,j} & \cdots & L_{i,M}\\
\vdots & \ddots & \vdots & \ddots & \vdots\\
L_{N,1} & \cdots & L_{N,j} & \cdots & L_{N,M}
\end{array}\right)\left(\begin{array}{c}
\alpha_{k}(x_{1},y_{1})\\
\alpha_{k}(x_{2},y_{2})\\
\vdots\\
\alpha_{k}(x_{j},y_{j})\\
\vdots\\
\alpha_{k}(x_{M},y_{M})
\end{array}\right)=L\alpha_{k}(x,y),\label{e:bkmatrix}
\end{equation}
for all $k=1,2,\ldots,W$, where $L$ denotes the matrix $L=(L_{i,j})_{i=1,j=1}^{N,M}$.
Thus, the fully-discretized modeling of the tomographic problem in
TAS results in solving the nonlinear system of equations

\begin{equation}
b^{k}=L\alpha_{k}(x,y),\quad k=1,2,\ldots,W,\label{e:systemb}
\end{equation}
where the vectors $b^{k}$ are known from measurements for all wavelengths
and the matrix \textbf{$L$} is calculated according to the beam arrangement.
In the past decades, numerous studies have made progress in this problem.
Next, we give a brief glance of it.

\subsection{Techniques for solving TAS problems}

\label{Sect:TAS2} Both nonlinear or linear approaches have been applied
in the literature to solve system \eqref{e:systemb}. Depending on
the approach for solving the tomographic problem, TAS can be divided
into \textbf{nonlinear TAS} and \textbf{linear TAS}. In nonlinear
TAS, all the equation systems are considered together and the inversion
is recast into a one-step optimization problem as \cite{DAI2018233}:

\begin{equation}
\underset{x,y\in\mathbb{R}^{M}}{\argmin}\mathop{\frac{1}{2}\sum_{k=1}^{W}}\|b^{k}-L\alpha_{k}(x,y)\|^{2}.\label{e:lsSA}
\end{equation}
The $x$ and $y$ distributions are directly solved from \eqref{e:lsSA}.
This optimization problem is usually solved by a global heuristic
optimization algorithm, such as \textbf{simulated annealing} (SA)
\cite{CAI2008250,DAI2018233}.

The SA algorithm was first proposed by Kirkpatrick et al. in 1983
\cite{Kirkpatrick671} and has been widely employed in large-scale
optimization problems \cite{LEITE1999545,CAI201011}. It is a heuristic
for finding the global minimum inspired by annealing in metallurgy.
A prominent advantage of SA is its insensitivity to the initial guesses
\cite{Ma:09}, while the major drawback is its high computational
cost. Besides, a priori information can be taken into consideration
\cite{DAI2018233}, which is another advantage of SA.

On the other hand, in linear TAS the problem is divided into two stages.
In the first stage, for every $k\in{\{1,2,\ldots,W\}}$, equation
\eqref{e:bkmatrix} is solved for each individual wavelength as a
linear equation system whose variables are the local absorption coefficients
$\alpha_{k}$. Classical algorithms, including \textbf{algebraic reconstruction
technique} (ART) \cite{GORDON1970471}, \textbf{maximum likelihood
expectation maximization} (MLEM) \cite{https://doi.org/10.1111/j.2517-6161.1977.tb01600.x}
and Tikhonov reconstruction \cite{DAUN2010105}, have been extensively
adopted for this stage.

For the second stage, absorption coefficients $a^{k}=(a_{j}^{k})_{j=1}^{M}\in\mathbb{R}^{M}$,
at each pixel $j$ and for each wavelength $k$, are supposed to have
been recovered in the first stage, i.e.,
\begin{equation}
a^{k}=\left(\begin{array}{c}
a_{1}^{k}\\
a_{2}^{k}\\
\vdots\\
a_{j}^{k}\\
\vdots\\
a_{M}^{k}
\end{array}\right)=\left(\begin{array}{c}
\alpha_{k}(x_{1},y_{1})\\
\alpha_{k}(x_{2},y_{2})\\
\vdots\\
\alpha_{k}(x_{j},y_{j})\\
\vdots\\
\alpha_{k}(x_{M},y_{M})
\end{array}\right),\quad k=1,2,\ldots,W.\label{e:systema}
\end{equation}
Now the properties $x$ and $y$ need to be solved according to the
nonlinear relationship between them and the absorption coefficients.
While, for every $k\in{\{1,2,\ldots,W\}}$, the absorption coefficients
$a^{k}$ have been recovered in the first stage, the $\alpha_{k}:\mathbb{R}\times\mathbb{R}\to\mathbb{R}$
become here operators. Commonly, in the literature, gas properties
within each pixel were calculated from the $a^{k}$s using a \textbf{nonlinear
fitting} (NF) method \cite{Grauer_2019}. This nonlinear fitting process
employs a trust-region-reflective algorithm \cite{doi:10.1137/0806023}
to find the least-squares of the discrepancy between $a^{k}$ and
$\alpha_{k}(x,y)$. The iterative process of this algorithm is briefly
described in Algorithm \ref{alg:NF}, see~\cite{doi:10.1137/0904038}.
In the algorithm, $g$ and $H$ refer to the first and second order
derivatives of the function to be solved, respectively, and $s=(s_{1},s_{2})$
is the step of the iteration. It should be noted that an analytical
expression of the nonlinear tomographic absorption spectroscopy problem
is not derivable. Here, the $g$ and $H$ are results of an approximation
of the nonlinear tomography formulation.

{\SetAlgoNoLine
\begin{algorithm}
\textbf{Initialization:} Set $\delta\in\mathbb{R}$. Choose $x_{j}^{0}\geq0$
and $y_{j}^{0}\geq0$\; set {$\ell=0$}\; \Repeat{set {$H=\nabla^{2}(\|a_{j}^{k}-\alpha_{k}(\cdot,\cdot)\|^{2})(x_{j}^{\ell},y_{j}^{\ell})$,
$g=\nabla(\|a_{j}^{k}-\alpha_{k}(\cdot,\cdot)\|^{2})(x_{j}^{\ell},y_{j}^{\ell})$}\;
set { $(s_{1},s_{2})=\argmin_{\|s\|\leq\delta}(\frac{1}{2}s^{T}Hs+g^{T}s)$}\;
\eIf{$\|a_{j}^{k}-\alpha_{k}(x_{j}^{\ell}+s_{1},y_{j}^{\ell}+s_{2})\|^{2}<\|a_{j}^{k}-\alpha_{k}(x_{j}^{\ell},y_{j}^{\ell})\|^{2}$}{
set {$x_{j}^{\ell+1}=x_{j}^{\ell}+s_{1}$, $y_{j}^{\ell+1}=y_{j}^{\ell}+s_{2}$}\; set {$\ell=\ell+1$}\;
adjust the trust region size $\delta$\; }{reduce the trust region
size $\delta$\;}}

\caption{The Nonlinear Fitting (NF) method applied in multi-spectral TAS.}
\label{alg:NF}
\end{algorithm}
}

A drawback of this method is that it still has a high computational
cost. Besides, since the algorithm obtains $x_{j}$ and $y_{j}$ pixel
by pixel, a priori information cannot be employed in this algorithm
for the process of obtaining temperature and concentration results
from absorption coefficients. That means, in traditional approaches
for TAS, a priori information is solely employed in the process to
solve the absorption coefficients, which is sometimes defective. For
example, if the temperature and the concentration satisfy different
a priori information, the method to add a priori information on absorption
coefficients cannot work. In addition, the calculation of temperature
and mole fraction from absorption coefficient can introduce extra
errors due to complicated error propagation. In the following sections,
we motivate the implementation of our proposed Algorithm \ref{a:abstract}
for tackling problems of the form given by \eqref{e:systema}, and
present a demonstrative example in which it outperforms the NF approach.

\subsection{Implementation of the descent pairs algorithm}

\label{sect:TAS3}Let $\mathbb{R}_{++}^{M}$ denote the positive orthant
of $\mathbb{R}^{M}.$ For every $k\in{\{1,2,\ldots,W\}}$, denote
by $\beta_{k}:\mathbb{R}_{++}^{M}\times\mathbb{R}_{++}^{M}\to\mathbb{R}_{++}^{M}$
the operator defined component-wise as
\begin{equation}
(\beta_{k}(x,y))_{j}:=\alpha_{k}(x_{j},y_{j}),\quad j=1,2,\ldots,M,
\end{equation}
where $\alpha_{k}:\mathbb{R}_{++}\times\mathbb{R}_{++}\to\mathbb{R}_{++}$
are the operators in \eqref{e:systema}. This yields the system of
equations
\begin{equation}
\beta_{k}(x,y)=a^{k},\quad k=1,2,\ldots,W.\label{e:systTAS}
\end{equation}
It is a known property of the absorption coefficients that the operators
$\alpha_{k}$ fulfill Assumption \ref{asumpt:1} \cite{CAI20171}.
Thus, there exist operators $\widetilde{\beta}_{k}:\mathbb{R}_{++}^{M}\to\mathbb{R}_{++}^{M}$
such that
\begin{equation}
\beta_{k}(x,y)={\rm diag}\left(\widetilde{\beta}_{k}(x)\right)y,\quad k=1,2,\ldots,W.\label{e:Eq8}
\end{equation}

In order to validate the implementation of the descent pairs algorithm
(Algorithm~\ref{a:abstract}) to this problem, we employ a property
of the TAS system of equations~(\ref{e:systTAS}) which is known
from the physics of the problem and is presented in the next remark.
\begin{rem}
For a fixed index $t\in{\{1,2,\ldots,W\}}$, define for each $k\in{\{1,2,\ldots,W\}}\setminus\left\{ t\right\} $
the functions
\begin{equation}
f_{k}(x):={\rm diag}\left(\widetilde{\beta}_{t}(x)\right)^{-1}\widetilde{\beta}_{k}(x).
\end{equation}
Then, it is known from the physics of the problem that the functions
$f_{k}$ are given component-wise by
\begin{equation}
(f_{k}(x))_{j}:=\frac{S_{k}}{S_{t}}\exp\left(-\left(E_{k}-E_{t}\right)\left(\frac{1}{x_{j}}-\frac{1}{T_{0}}\right)\right),\quad j=1,2,\ldots,M,
\end{equation}
where $S_{k}$, $S_{j}>0$ and $E_{k},E_{t}\in\mathbb{R}$ are constants
which depend on the wavelength, and $T_{0}$ is a positive known constant.
\end{rem}

The following lemmata assure that Assumption~\ref{assumpt:2} holds
for the system of equations (\ref{e:systTAS}).

\begin{lemma}\label{l:dd1} Let $t$ be an index such that $E_{t}=\min\left\{ E_{1},E_{2},\ldots,E_{W}\right\} $
and let $\bar{x}\in\mathbb{R}_{++}^{M}$be a fixed arbitrary vector.
Then, for any $k\in{\{1,2,\ldots,W\}\setminus{\left\{ t\right\} }}$,
the vector given by
\begin{equation}
v^{k}:={\rm diag}(a^{t})^{-1}a^{k}-f_{k}(\bar{x})
\end{equation}
is a descent direction at the point $\bar{x}$ for the function $g_{k}:\mathbb{R}_{++}^{M}\to\mathbb{R}_{+}$
defined as
\begin{equation}
g_{k}(\cdot):=\frac{1}{2}\|f_{k}(\cdot)-{\rm diag}(a^{t})^{-1}a^{k}\|^{2}.
\end{equation}
\end{lemma}
\begin{proof}
The partial derivative of the $j$-th component of $f_{k}$ is given
by
\begin{equation}
\frac{\partial\left(f_{k}(x)\right)_{j}}{\partial x_{\ell}}=\left\{ \begin{array}{ll}
\frac{S_{k}}{S_{t}}\left(E_{k}-E_{t}\right)\frac{\left(f_{k}(x)\right)_{j}}{x_{j}^{2}}, & \;\text{if}\;\ell=j,\\
0, & \;\text{otherwise}.
\end{array}\right.
\end{equation}
Therefore, the gradient of $g_{k}$ at the point $\bar{x}\in\mathbb{R}_{++}^{M}$
is given by
\begin{equation}
\nabla g_{k}(\bar{x})=-\frac{S_{k}}{S_{t}}\left(E_{k}-E_{t}\right){\rm diag}\left(\left(\frac{(f_{k}(\bar{x}))_{j}}{\bar{x}_{j}^{2}}\right)_{j=1}^{M}\right)v^{k}.
\end{equation}
Hence, we have
\begin{equation}
\nabla g_{k}(\bar{x})^{^{T}}v^{k}=-\frac{S_{k}}{S_{t}}\left(E_{k}-E_{t}\right){\rm diag}\left(\left(\frac{(f_{k}(\bar{x}))_{j}}{\bar{x}_{j}^{2}}\right)_{j=1}^{M}\right)\|v^{k}\|^{2}<0,\label{e:dd}
\end{equation}
since $S_{k},S_{t}>0$, $E_{k}-E_{t}>0$ for all $k$ and all diagonal
elements of the matrix are positive.
\end{proof}
\begin{lemma}\label{l:dd2} Let $\bar{x},\bar{y}\in\mathbb{R}_{++}^{M}$.
The vector $w^{k}$ given by
\begin{equation}
w^{k}:=a^{k}-\beta_{k}(\bar{x},\bar{y})=a^{k}-{\rm diag}\left(\widetilde{\beta}_{k}(\bar{x})\right)\bar{y}
\end{equation}
is a descent direction at the point $\bar{y}$ for the function $h_{k}:\mathbb{R}_{++}^{M}\to\mathbb{R}_{+}$
defined by
\begin{equation}
h_{k}(y):=\frac{1}{2}\|a^{k}-\beta_{k}(\bar{x},y)\|^{2}={\displaystyle \frac{1}{2}\|a^{k}-{\rm diag}\left(\widetilde{\beta}_{k}(\bar{x})\right)y\|^{2},\quad\text{for all }y\in\mathbb{R}_{++}^{M}.}
\end{equation}
\end{lemma}
\begin{proof}
The gradient of $h_{k}$ at the point $\bar{y}$ is given by
\begin{equation}
\nabla h_{k}(\bar{y})=-{\rm diag}\left(\widetilde{\beta}_{k}(\bar{x})\right)w^{k}.
\end{equation}
Thus, we have
\begin{equation}
\nabla h_{k}(\bar{y})^{T}w^{k}=-{\rm diag}\left(\widetilde{\beta}_{k}(\bar{x})\right)\|w^{k}\|^{2}<0,
\end{equation}
where the last strict inequality holds since $\widetilde{\beta}_{k}:\mathbb{R}_{++}^{M}\to\mathbb{R}_{++}^{M}$.
\end{proof}
The above analysis shows that the system of equations \eqref{e:systTAS}
fulfills Assumptions~\ref{asumpt:1} and~\ref{assumpt:2}, making
Algorithm \ref{a:abstract} a good choice for handling it. In Section
\ref{sect:TAS4} we provide a numerical demonstration of its good
performance.

\subsection{Improving the descent pairs algorithm by superiorization\label{subsec:Improving-the-descent}}

The iterative nature of Algorithm \ref{a:abstract} enables us to
introduce a priori conditions into the iterative process. We do this
via the superiorization methodology.

\subsubsection{The superiorization methodology\label{subsec:superiorization}}

The superiorization methodology \cite{herman2012superiorization}
works by taking an iterative algorithm, investigating its \textbf{perturbation
resilience}, and then, using proactively such permitted perturbations,
it forces the perturbed algorithm to do something useful in addition
to what it was originally designed to do. The original unperturbed
algorithm is called the \textbf{basic algorithm} and the perturbed
algorithm is called the \textbf{superiorized version of the basic
algorithm}.

When the basic algorithm is computationally efficient and useful in
terms of the application at hand, and the perturbations are simple
and not expensive to calculate, then the advantage of this methodology
is that, for essentially the computational cost of the basic algorithm,
we are able to get something more by steering its iterates according
to the perturbations. A detailed description of the superiorization
methodology along with pertinent up-to-date references can be found
in several papers, see, e.g., \cite{DFS2019}, \cite{Herman2019-degruyter}

This general principle has been successfully used in a variety of
important practical applications, see the recent papers in the, compiled
and continuously updated, bibliography of scientific publications
on the superiorization methodology and perturbation resilience of
algorithms \cite{sup-bib}, where many applications oriented works
with the method appear, e.g., \cite{multi-cast-2021}.

In the language of \cite{Censor-weak-2015}, the superiorization methodology
allows us to employ a given function, called \textbf{target function},
$\varphi:\mathbb{R}^{M}\rightarrow\mathbb{R}.$ It interlaces into
the iterations of a basic algorithm steps that perform locally reductions
of the target function, these steps are called \textbf{perturbations}.
The resulting \textbf{superiorized version of the basic algorithm}
is expected to retain the convergence properties of the basic (unperturbed)
algorithm but, additionally, steer the process to an output with reduced,
not necessarily minimized, value of the target function.

We use this methodology to superiorize the descent pairs algorithm
(Algorithm \ref{a:abstract}), as we describe below.

\subsubsection{Implementation of the superiorized version of the descent pairs algorithm\label{subsec:Implement}}

In our implementations, priors are applied to Algorithm \ref{a:abstract}
according to the superiorization methodology. Corresponding target
functions $\varphi$, mentioned above, to lead the function reduction
perturbations in the superiorization process are the priors defined
by \cite{Shui:21}. These are the \textbf{total variation} (TV) function
and the \textbf{smoothness} (Tikhonov, Tik for short) function \cite{DAUN2010105},
defined, respectively, by

\begin{equation}
\psi_{TV}(Z):=\sum_{i,j}\sqrt{(Z(i,j)-Z(i+1,j))^{2}+(Z(i,j)-Z(i,j+1))^{2}},\label{eq:TV}
\end{equation}

\noindent and

\begin{equation}
\psi_{Tik}(Z):=\mathop{\sum_{i,j}}\left(Z(i,j)-\frac{1}{rn}\mathop{\sum_{ii,jj}}\left(Z(ii,jj)-Z(i,j)\right)\right)^{2},\label{Tik}
\end{equation}
where $Z$ represents one of the properties $x$ or $y$, $\{(i,j)\}_{i=1,j=1}^{n,n}$
is a set of pairs of indices of pixels in an exact grid in the ROI
(see Figure \ref{fig:1}), $rn$ is the number of pixels that surround
pixel $(i,j)$, and $(ii,jj)$ are the indices of these pixels. The
pseudo-code of this method applied to \eqref{e:systTAS} is shown
in Algorithm \ref{a:SUPAFP}. In the experimental runs the function
$\varphi$ in the algorithm will be either $\psi_{TV}(Z)$ or $\psi_{Tik}(Z)$
of (\ref{eq:TV}) or (\ref{Tik}), respectively. Although in the experiments
only the above two functions are chosen as target functions, and thus
gradients are involved in the pseudo-code below, it should be emphasized
that in general any target function can be considered and function
reduction is not necessarily via derivatives \cite{DFS2019,censor2019derivative,gibali2018dc}.
To avoid zero as the divisor when taking gradient for the TV function,
we add a small constant $10^{-5}$ under the square root sign in our
application.\smallskip{}

{\SetAlgoNoLine
\begin{algorithm}
\textbf{Initialization:} Set $\lambda_{x}$, $\lambda_{y}>0$, $\eta_{x}$
and $\eta_{y}$. Set the index $t$ such that $E_{t}=\min\left\{ E_{1},E_{2},\ldots,E_{W}\right\} $
and choose $x^{0}$ and $y^{0}$ in $\mathbb{R}_{+}^{M}$\; set {$\ell=0$}\;
\Repeat{ set {$x^{\ell,0}=x^{\ell}$}\; \For{$q=1,2,\ldots,W$}{
set {$v_{x}:=-\frac{\nabla\varphi(x^{\ell,q-1})}{\|\nabla\varphi(x^{\ell,q-1})\|}$}\;
\While{$\varphi(x^{\ell,q-1}+\beta_{x}v_{x})>\varphi(x^{\ell,q-1})$}{
set {$\eta_{x}=\gamma\eta_{x}$}\; } set {$z_{x}=x^{\ell,q-1}+\eta_{x}v_{x}$}\;
set {$x^{\ell,q}=z_{x}+\lambda_{x}\left({\rm diag}(a^{t})^{-1}a^{q}-{\rm diag}\left(\widetilde{\beta}_{t}(z_{x})\right)^{-1}\widetilde{\beta}_{q}(z_{x})\right)$}\;
} set {$x^{\ell+1}=x^{\ell,W}$}\; set {$y^{\ell,0}=y^{\ell}$}\;
\For{$q=1,2,\ldots,W$}{ set {$v_{y}:=-\frac{\nabla\varphi(y^{\ell,q-1})}{\|\nabla\varphi(y^{\ell,q-1})\|}$}\;
\While{$\varphi(y^{\ell,q-1}+\beta_{y}v_{y})>\varphi(y^{\ell,q-1})$}{
set {$\eta_{y}=\gamma\eta_{y}$}\; } set {$z_{y}=y^{\ell,q-1}+\eta_{y}v_{y}$}\;
set {$y^{\ell,q}=z_{y}+\lambda_{y}\left(a^{q}-\beta_{q}(x^{\ell+1},z_{y})\right)$};
} set {$y^{\ell+1}=y^{\ell,W}$}\; } \caption{The Superiorized version of the descent pairs algorithm (Algorithm
\ref{a:abstract}), nicknamed SUP-DPA, applied in multi-spectral TAS.}
\label{a:SUPAFP}
\end{algorithm}

In this algorithm, $\lambda$ is the step length of the DPA iteration
and $\eta$ is the step length for superiorization, which is contracted
according to $\gamma$ as the iteration processes. The value of $\gamma$
and the initial value of $\eta$ influence the impact of the superiorization
together. We denote by $v$ the normalized gradient of the function
$\varphi$ at the current iteration and $z$ is an intermediate variable
between the superiorization and the DPA iteration.

\subsection{A numerical experiment\label{subsec:experiment}}

\label{sect:TAS4} In this subsection, Algorithms~\ref{a:abstract},
\ref{alg:NF} and \ref{a:SUPAFP} are nicknamed as Algorithm DPA,
NF and SUP-DPA, respectively. We conducted simulations on numerous
temperature and concentration phantoms, two of which are shown here
as demonstrative examples. Phantom 1 mimics the temperature and H$_{2}$O
concentration in a two-dimensional McKenna flame. Phantom 2 imitates
smooth temperature and concentration distributions consisting of two
Gaussian peaks.

The square ROI was divided into 40\texttimes 40 pixels grids and 160
beams were employed, which were distributed uniformly and parallelly
in four directions, $0^{\circ}$, $45^{\circ}$, $90^{\circ}$ and
$135^{\circ}$. Ten discrete wavelengths were selected from the H$_{2}$O
absorption spectrum.

In order to observe the progress and behavior of the DPA algorithm
(Algorithm~\ref{a:abstract}) for solving the second stage of the
multi-spectral TAS problem, as explained in Subsection \ref{Sect:TAS2}
above, we compared its performance with the NF algorithm (Algorithm
\ref{alg:NF}). In our implementation with MATLAB, the nonlinear least-squares
solver, i.e., the function ``lsqnonlin'', was applied as the NF
algorithm, wherein MATLAB's trust-region algorithm was employed in
this function to find the solution. For the solution of the first
stage we employed the superiorized ART algorithm \cite{4407758,Davidi2009perturbation}.

To mimic practical situations, noise was added to the absorbance measurements.
Uniform noise was used in our simulations, defined as
\begin{equation}
b_{i;mes}^{k}:=b_{i;ori}^{k}\cdot(1+rand\cdot u),
\end{equation}
where the subscript $ori$ refers to the original absorbance calculated
from the phantom without noise, while $mes$ refers to the practical
absorbance with noise added; $u$ is the noise level; $rand$ is a
random number in the interval $(-1,1)$.

We tested our algorithms on two different phantoms, Phantom 1 and
Phantom~2. For both phantoms, the DPA algorithm was initialized with
$\lambda_{x}=1000$, $\lambda_{y}=2$, and $x^{0}$ and $y^{0}$ randomly
obtained by picking their components between some lower and upper
bounds. These bounds were set, for all $j=1,2,\ldots,M$, as $x_{j}^{0}\in(400,2000)$
and $y_{j}^{0}\in(0.005,0.2)$ in the algorithmic runs for the reconstruction
of Phantom~1, while for all $j=1,2,\ldots,M$, $x_{j}^{0}\in(800,2400)$
and $y_{j}^{0}\in(0.005,0.2)$ in the algorithmic runs for the reconstruction
of Phantom~2.

For the SUP-DPA algorithm (Algorithm \ref{a:SUPAFP}), in the reconstruction
of Phantom~1, TV was employed as the prior, i.e., as the function
$\varphi$ in the algorithm. The SUP-DPA algorithm was initialized
with $\lambda_{x}=1000$, $\lambda_{y}=2$, $\beta_{x}=5\times10^{6}$,
$\beta_{y}=10$ and $\gamma=0.999$. Initial guesses of $x^{0}$ and
$y^{0}$ were vectors randomly chosen in the intervals described above
for the DPA algorithm.

For the reconstruction of Phantom 2 with the SUP-DPA algorithm, smoothness
was regarded as the prior, thus, the function $\varphi$ in the algorithm
was chosen as $\psi_{Tik}$ of (\ref{Tik}) and the algorithm was
initialized with $\lambda_{x}=1000$, $\lambda_{y}=2$, $\beta_{x}=5\times10^{4}$,
$\beta_{y}=10$ and $\gamma=0.999$. Initial guesses of $x^{0}$ and
$y^{0}$ were vectors randomly picked in the intervals described above
for the DPA algorithm.

All runs of the DPA and the SUP-DPA algorithms were stopped when either
the number of iterations exceeded 50 or when $\sum_{k=1}^{W}\|b^{k}-\beta_{k}(x^{\ell+1},y^{\ell+1})\|<10^{-3}$.
Figure \ref{fig:stop} shows the relation between relative error in
$x$ and the stopping criterion function $\sum_{k=1}^{W}\|b^{k}-\beta_{k}(x^{\ell+1},y^{\ell+1})\|$
as the iterations proceed for DPA and SUP-DPA. The relative errors
are defined as

\begin{equation}
\begin{cases}
Error_{x}:=\frac{\left\Vert x_{rec}-x_{act}\right\Vert _{2}}{\left\Vert x_{act}\right\Vert _{2}},\\
Error_{y}:=\frac{\left\Vert y_{rec}-y_{act}\right\Vert _{2}}{\left\Vert y_{act}\right\Vert _{2}},
\end{cases}\label{eq:error}
\end{equation}
where the subscripts ``rec'' and ``act'' stand for reconstructed
value and actual value, respectively.

It can be seen that after the stopping criterion function value fell
below $10^{-3}$, the relative error hardly changes, indicating that
a convergence is reached. The relative error of SUP-DPA had a slight
fluctuation at the left end of the curve because of the perturbation
introduced. On the other hand, 50 iterations is a large enough number
for this testing case to assume convergence when the stopping criterion
function value cannot decrease to $10^{-3}$. For the NF algorithm,
the function tolerance and optimality tolerance were both set to $5\times10^{-10}$.
When implementing the NF algorithm, $\delta$ was initially set as
the 2-norm of the difference between the upper and lower bounds, ($(800,2400)$
for $x$ and $(0.005,0.2)$ for $y)$, and the adjustment of $\delta$
followed the default settings in the MATLAB function ``lsqnonlin''.
The algorithm was terminated when $\delta<5\times10^{-10}$.

\begin{figure}[H]
\centering\includegraphics[scale=0.45]{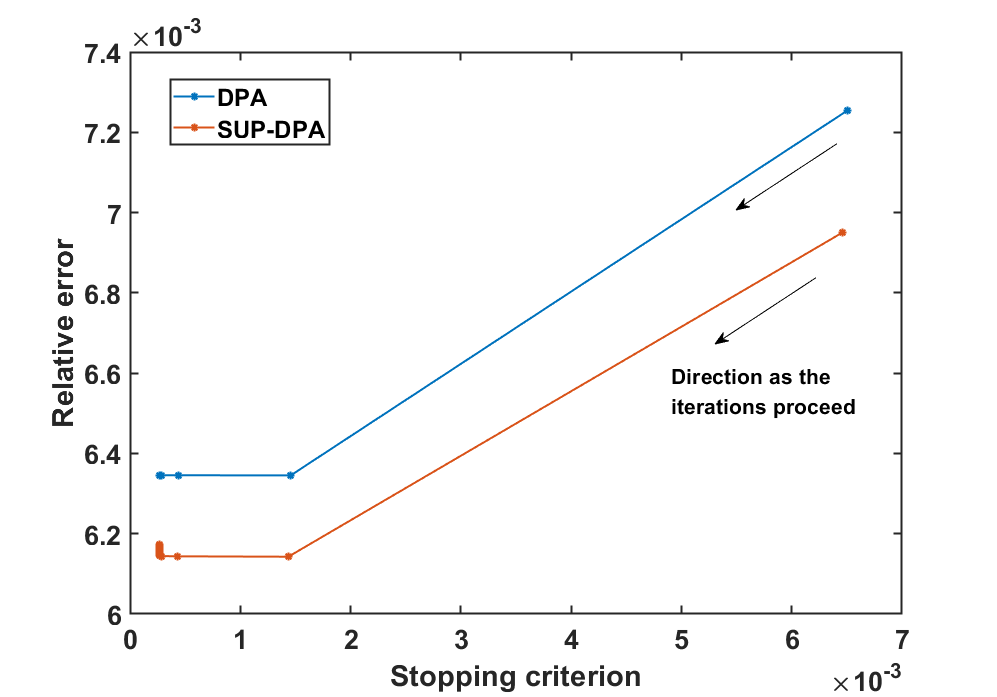}\caption{Relation between relative error and stopping criterion as the iterations
proceed.}

\label{fig:stop}
\end{figure}

Results are shown in Figures~\ref{fig:2} and~\ref{fig:3} and Table~\ref{table1}.
The name of the algorithms is shown at the top of the corresponding
reconstructed profiles. The relative reconstruction errors are also
labeled, which are defined by (\ref{eq:error}). To be specific, the
unit of temperature $x$ was $K$, which is labeled beneath its color-bar
in the figures, and for the mole fraction $y$, the unit is one, thus,
no unit was labeled.

\begin{table}[H]
\caption{Computation times of the algorithms}
\label{table1} \centering{}%
\begin{tabular}{|c|c|c|c|}
\hline
Computation time (sec)  &
DPA  &
SUP-DPA  &
NF\tabularnewline
\hline
\hline
Phantom 1  &
0.48  &
0.56  &
9.16\tabularnewline
\hline
Phantom 2  &
0.48  &
0.54  &
7.67\tabularnewline
\hline
\end{tabular}
\end{table}

Figure \ref{fig:2} shows Phantom 1 and the results recovered by the
DPA, the SUP-DPA and the NF algorithms. Temperature and concentration
profiles generated by the three algorithms show similar shapes to
those of Phantom 1, although there are some defects at the edges of
the profiles. The reconstruction errors of SUP-DPA is smaller than
those of the DPA and the NF algorithm. Besides, it takes much less
time for the DPA and the SUP-DPA algorithms to finish the reconstruction
than it took for the NF algorithm. Overall, the SUP-DPA algorithm
can be regarded as the best among the three algorithms, for the particular
experiments that we conducted and report here.

\begin{figure}[ht!]
\centering\includegraphics[scale=0.45]{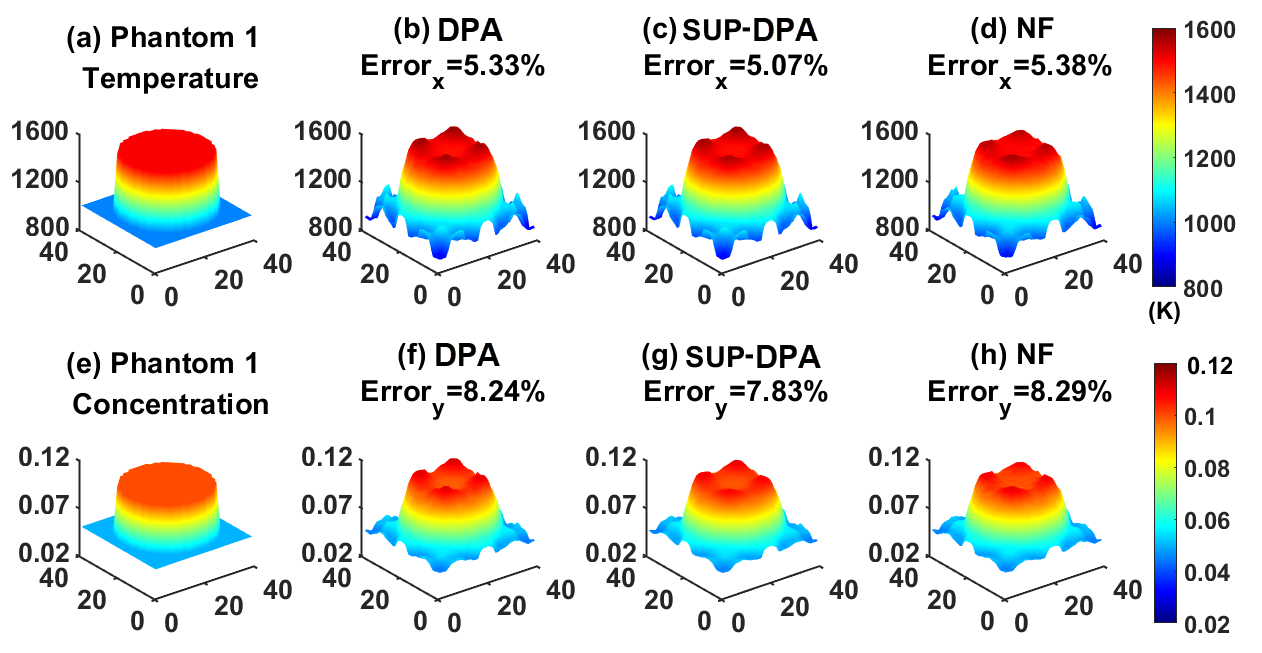}

\caption{(a) Temperature profile of Phantom 1. Temperature profiles recovered
by (b) the DPA algorithm, (c) the SUP-DPA algorithm, and (d) the NF
algorithm. (e) Concentration profile of Phantom 1. Concentration profiles
recovered by (f) the DPA algorithm, (g) the SUP-DPA algorithm, and
(h) the NF algorithm.}
\label{fig:2}
\end{figure}

Figure \ref{fig:3} shows Phantom~2 and the results recovered by the
DPA, the SUP-DPA and the NF algorithms. Results by the three methods
show a similar shape to Phantom~2. The error of the DPA algorithm
is slightly smaller than that of the NF algorithm, while the SUP-DPA
algorithm improves the error of both other algorithms. Besides, the
computational time of the SUP-DPA algorithm is much smaller than that
of the NF algorithm. Though the DPA algorithm has the lowest computation
time, its reconstruction effect is worse than that of the SUP-DPA
algorithm. Therefore, the SUP-DPA algorithm can still be regarded
advantageous over the DPA and the NF algorithms.

\begin{figure}[ht!]
\centering \includegraphics[scale=0.45]{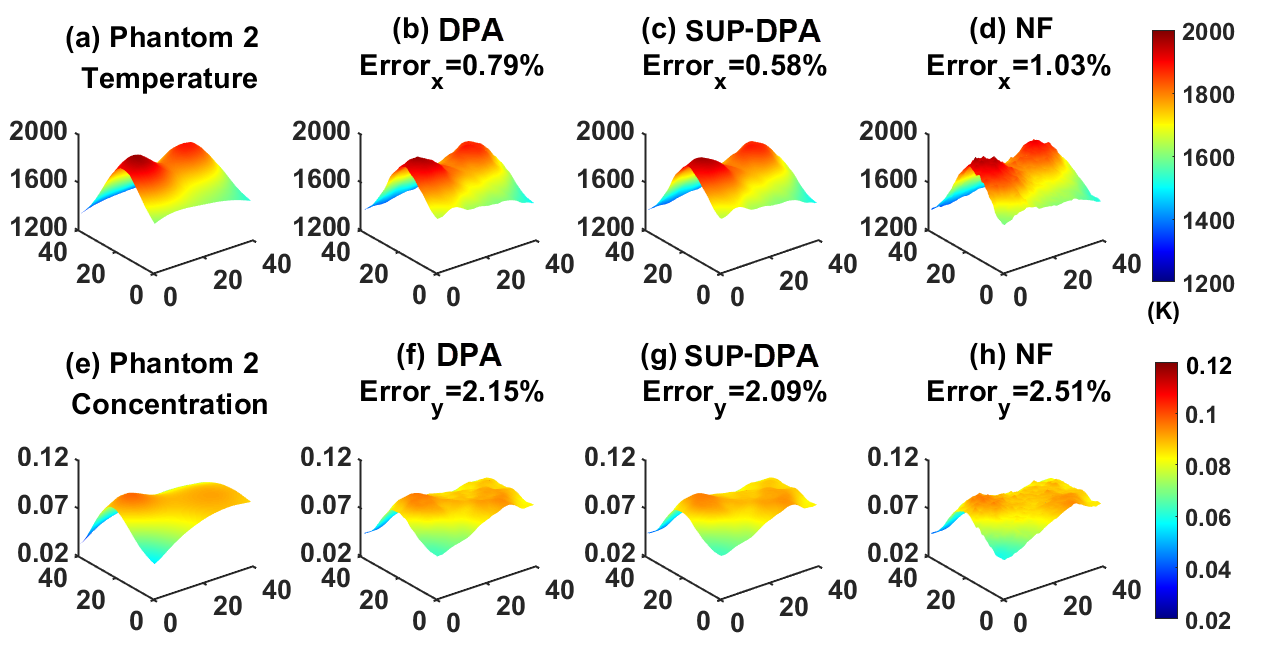}\caption{(a) Temperature profile of Phantom 2. Temperature profiles recovered
by (b) the DPA algorithm, (c) the SUP-DPA algorithm, and (d) the NF
algorithm. (e) Concentration profile of Phantom 2. Concentration profiles
recovered by (f) the DPA algorithm, (g) the SUP-DPA algorithm, and
(h) the NF algorithm.}
\label{fig:3}
\end{figure}

To further investigate the performance of the algorithms, simulation
studies were conducted under different spatial resolutions. For Phantom
2, reconstructions were conducted when it was divided into $20\times20$,
$40\times40$, $60\times60$ and $80\times80$ pixel grids, denoted
as ``gridding scale'' 20, 40, 60 and 80, respectively. To ensure
the comparability of the results, the measurement of each case was
conducted in four directions, while the number of beams in each direction
increased correspondingly to the gridding scale, which means that
for a $G\times G$ grid, $4\times G$ measurement beams were applied.
Other conditions, including noise level, parameters and stopping criteria,
were the same as those described in Figure \ref{fig:3}. The results
are shown in Figure \ref{fig:grid}.

As the pixel grid was made finer, the computation time grew exponentially,
indicating an increasing computational efficiency improvement of DPA
and SUP-DPA compared with NF. In each simulation condition, DPA and
SUP-DPA had better performance than NF, with SUP-DPA smaller than
DPA in relative error and DPA slightly smaller in computation time.
At gridding scale 20, there were the largest errors. This might be
due to less measurement beams. Coarser gridding made the phantom less
smooth, which could result in difficulties for the reconstruction.

\begin{figure}[H]
\centering\includegraphics[scale=0.32]{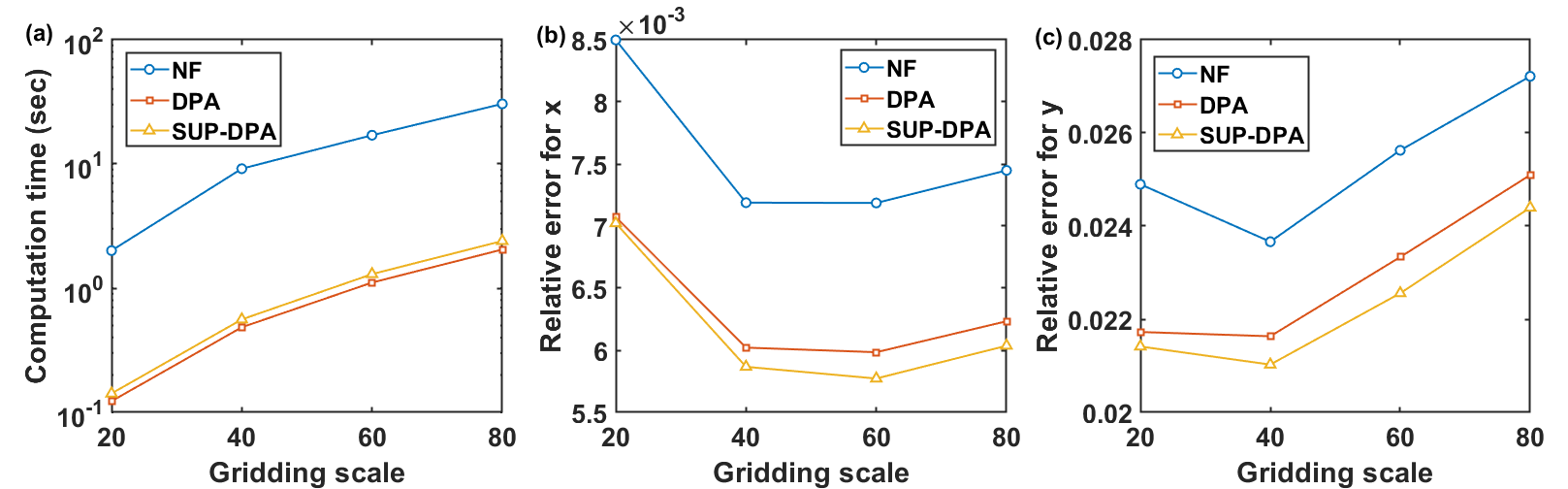}\caption{(a) Computation time, (b) reconstruction error for $x$, and (c) reconstruction
error for $y$, with respect to the gridding scale.}

\label{fig:grid}
\end{figure}

\bigskip{}

\paragraph{Acknowledgements}

F.J.A.A. and D.T.B. were partially supported by Grants PGC2018-097960-B-C22
and PID2022-136399NB-C21 funded by ERDF/EU and by
MICIU/AEI/ 10.13039/501100011033.
D.T.B was supported by Grant PRE2019-09075 funded by
MICIU/AEI/10.13039/501100011033 and by ``ESF Investing in your future''. F.J.A.A. was
partially supported by the Generalitat Valenciana (AICO/2021/165).
The work of Y.C. was supported by the ISF-NSFC joint research program
grant No. 2874/19, by the U.S. National Institutes of Health (NIH)
Grant No. R01CA266467 and by the Cooperation Program between the German
Cancer Research Center (DKFZ) and Israel’s Ministry of Innovation,
Science and Technology (MOST). W. C. and C. S. were partially supported
by the National Natural Science Foundation of China (Grant No. 51976122,
52061135108).

\section*{Declarations}

\paragraph{Conflict of interest}

The authors declare no competing interests. \newpage{}

\bibliographystyle{plain}
\bibliography{afp-revised-040524}

\end{document}